\documentclass[11pt]{amsart}
\usepackage{graphicx}
\usepackage[active]{srcltx}
 \makeatletter
\renewcommand*\subjclass[2][2000]{%
  \def\@subjclass{#2}%
  \@ifundefined{subjclassname@#1}{%
    \ClassWarning{\@classname}{Unknown edition (#1) of Mathematics
      Subject Classification; using '1991'.}%
  }{%
    \@xp\let\@xp\subjclassname\csname subjclassname@#1\endcsname
  }%
}
 \makeatother
\usepackage{enumerate,url,amssymb,  mathrsfs}

\newtheorem{theorem}{Theorem}[section]

\newtheorem*{lemma*}{Lemma}

\newtheorem{corollary}[theorem]{Corollary}

\theoremstyle{definition}

\newtheorem{example}[theorem]{Example}

\theoremstyle{remark}
\newtheorem{remark}[theorem]{Remark}

\numberwithin{equation}{section}

\def\XXint#1#2#3{{\setbox0=\hbox{$#1{#2#3}{\int}$}
\vcenter{\hbox{$#2#3$}}\kern-.5\wd0}}

\def\le{\leqslant}
\def\ge{\geqslant}
\begin{document}

\title{Sharp pointwice estimate for Fock spaces}  \subjclass{Primary  	30D15;
Secondary 46E15}



 \author{Friedrich Haslinger}
\thanks{The first-name author was partially supported by the Austrian Science Fund, FWF-Projekt P 28154.}
\address{Faculty of Mathematics
University of Vienna,  Faculty of
Mathematics, Oskar-Morgenstern-Platz 1, 1090 Wien, Austria}

 \email{friedrich.haslinger@univie.ac.at}
\author{ David Kalaj}
\address{Faculty of Natural Sciences and
Mathematics, University of Montenegro, Cetinjski put b.b. 81000
Podgorica, Montenegro} \email{davidk@ucg.ac.me}
\author{Djordjije Vujadinovi\'c}
\address{Faculty of Natural Sciences and
Mathematics, University of Montenegro, Cetinjski put b.b. 81000
Podgorica, Montenegro} \email{djordjijevuj@ucg.ac.me}
\begin{abstract}
 Firstly we establish a sharp pointwise estimate for the arbitrary derivative of the function $f\in F_{\alpha}^{p},$ where $F_{\alpha}^{p}$ denotes the Fock space for $1\leq p<\infty.$  Then, in a particular Hilbert case when $p=2$ we establish another specific pointwise sharp estimate. We also consider the differential operator between $F_{\alpha}^{p}$  and $F_{\beta}^{p}$ for $\beta>\alpha$ and its adjoint.
\end{abstract}
\maketitle
\section{Introduction}
Let $\mathbb{C}$ be as usual the complex plane and by $dA(z)(=dxdy)$ we denote the Lebesgue measure on the complex plane. Through the paper for any positive parameter $\alpha$ we consider the Gaussian-probability measure
$$d\mu_{\alpha}(z)=\frac{\alpha}{\pi}e^{-\alpha|z|^{2}}dA(z).$$

For $1\leq p<\infty,$ $L^{p}(\mathbb{C},d\mu_{\alpha})$ denotes the space of all Lebesgue measurable functions $f$ on $\mathbb{C}
$ such that
$$\|f\|_{p,\alpha}^{p}=\frac{p\alpha}{2\pi}\int_{\mathbb{C}}|f(z)|^{p}e^{-\frac{p\alpha|z|^{2}}{2}}dA(z)<\infty.$$
In other words, $f\in L^{p}(\mathbb{C},d\mu_{\alpha})$ if and only if $f(z)e^{-\frac{\alpha|z|^{2}}{2}}\in L^{p}(\mathbb{C},dA).$

 The Segal--Bargmann space  also known as the Fock space, denoted by $F_{\alpha}^{2},$ consists of all entire functions $f$ in $L^{2}(\mathbb{C},d\mu_{\alpha}).$

 Although the Fock space can be defined on $\mathbb{C}^{n},$ $n>1,$ the main work of this paper is developed in the context of one complex variable.  The appropriate material for the Fock space of finitely many complex variables reader may find in \cite{foland}, \cite{hall} and \cite{janson}.

 Further, we extend the earlier introduced notion of Fock space for $1\leq p<\infty.$
 Namely, for $1\leq p<\infty$ by $F_{\alpha}^{p}$ we denote the Fock space which consists of entire functions in $L^{p}(\mathbb{C},d\mu_{\alpha}).$
For any $p\geq 1$ the Fock space $F_{\alpha}^{p}$ is closed in $L^{p}(\mathbb{C},d\mu_{\alpha}),$ and from the same reason $F_{\alpha}^{p}$ is a Banach space.

 We should mention that the space $F_{\alpha}^{\infty}$ is defined to be the space of all entire functions $f$ such that
 $$\|f\|_{\infty,\alpha}=\sup\{|f(z)|e^{-\frac{\alpha|z|^{2}}{2}}:z\in\mathbb{C}\}<\infty.$$

Since  $L^{2}(\mathbb{C},d\mu_{\alpha})$ is the Hilbert space with the inner product $$\left<f,g\right>=\int_{\mathbb{C}}f(z)\overline{g(z)}d\mu_{\alpha}(z),$$
the Fock space $F_{\alpha}^{2}$ as its closed subspace determines a natural orthogonal projection $P_{\alpha}: L^{2}(\mathbb{C},d\mu_{\alpha}) \rightarrow F_{\alpha}^{2}.$

It can be shown (see \cite{zhu}) that $P_{\alpha}$ is an integral operator induced with the reproducing kernel $$K_{\alpha}(z,w)=e^{\alpha z\bar{w}}.$$
More precisely, $$P_{\alpha}f(z)=\int_{\mathbb{C}}K_{\alpha}(z,w)f(w)d\mu_{\alpha}(w), f\in L^{2}(\mathbb{C},d\mu_{\alpha}),$$ and specially
\begin{equation}
\label{hilbert}f(z)=\int_{\mathbb{C}}K_{\alpha}(z,w)f(w)d\mu_{\alpha}(w), f\in F_{\alpha}^{2}.\end{equation}
For the problem related to the question of boundedness for the operator $P_{\alpha}$ on $F_{\alpha}^{p}$ we refer to \cite{domi1}.

Normalized reproducing kernel at the point $z$ is given by the sequel formula $$k_{z}(w)=K_{\alpha}(w,z)K_{\alpha}^{-\frac{1}{2}}(z,z)=e^{\alpha w\bar{z}-\alpha|z|^{2}}.$$
 The reproducing property of the kernel $K_{\alpha}(z,w)$ implies that $k_{z}(w)$ is the unit vector in $F_{\alpha}^{2},$ and surprisingly this property of $k_{z}(w)$ stays valid in any Fock space $F_{\alpha}^{p},$ $1\leq p\leq \infty.$

At this point we should underline that for every entire harmonic function $f,$ the function $\varphi(|f|),$ where $\varphi$ is some real convex function,  is subharmonic in any region of the complex plane. This fact will play a significant role in proving the main results of this paper.

A property that  characterizes  the  subharmonic  mappings $u$  is the  sub-meanvalue property which states that if $u$ is a subharmonic function defined on a domain $\Omega,$ then for every closed disk $\overline{D}(z_{0},r)\subset \Omega$ we have the following inequality
$$u(z_{0})\leq \frac{1}{2\pi r}\int_{|z-z_{0}|=r}u(z)|dz|.$$

\section{ Statement of the main results}
\subsection{The optimal rate growth for the functions in $F_{\alpha}^{p}$}

According to the Theorem 2.7. from \cite{zhu} for any $1\leq p\leq\infty$ and $z\in\mathbb{C}$ the following inequality is satisfied
\begin{equation}\label{zhuhe}|f(z)|\leq e^{\frac{\alpha|z|^{2}}{2}}\|f\|_{p,\alpha}, f\in F_{\alpha}^{p}.\end{equation}
Moreover, equality is attained for the functions of the form $$f(w)=e^{\alpha w\bar{z}-\frac{\alpha|z|^{2}}{2}+i\theta},$$ where $\theta$ is a real number.

 The main goal of this paper  is to determine the optimal rate growth of the arbitrary derivative for the functions in Fock spaces $F_{\alpha}^{p},$ $1\leq p<\infty.$

 Let us mention that a similar problem was treated in a different framework such as analytic Hardy space, harmonic Bergman space in the unit disk, et cetera (see for instance \cite{Axler},\cite{kalaj},\cite{kalaj1} and \cite{macinture}). We also refer to the paper \cite{carlen} for some multi-dimensional approach.

Our first main result reads as follows.

\begin{theorem}\label{main}
Let $1\leq p<\infty$ and $f\in F_{\alpha}^{p}.$ For $n\in \mathbf{N}$  let $$T_n(f,z)(w)=\sum_{k=0}^{n-1} \frac{f^{k}(z)}{k!} (w-z)^k$$ and
$$c_n(|z|)=\frac{ (\alpha p/2)^{n/2} n!}{ \Gamma^{1/p}\left(1+\frac{n p}{2}\right)}e^{\frac{\alpha |z|^2}{2}}.$$ Then the following sharp pointwise estimate holds

$$|f^{(n)}(z)|\le c_n(|z|) \|f-T_n(f,z)\|_{p,\alpha}.$$ The equality is attained in a certain $z$ for the mapping $$f(w) = e^{\alpha (z-w)\bar z} (w-z)^n.$$

In particular for $n=1$ we have the sharp estimate

\begin{equation}\label{dera}|f'(z)|\le \frac{ (\alpha p/2)^{1/2} }{ \Gamma^{1/p}\left(1+\frac{ p}{2}\right)}e^{\frac{\alpha |z|^2}{2}}\|f(\cdot)-f(z)\|_{p,\alpha}.\end{equation}

\end{theorem}
\begin{remark}
Let us notice that Theorem \ref{main} in a special case when $n=0$ gives the main result of \eqref{zhuhe} (Theorem 2.7. from \cite{zhu}). Moreover, combining \eqref{zhuhe} and \eqref{dera} we get the following estimate  \begin{equation}\label{dera1}|f'(z)|\le \left(\frac{ (\alpha p/2)^{1/2} }{ \Gamma^{1/p}\left(1+\frac{ p}{2}\right)}\right)\left(e^{\frac{\alpha |z|^2}{2}}+e^{{\alpha |z|^2}}\right)\|f\|_{p,\alpha}.\end{equation}
It would be of interest to find the best estimates in \eqref{dera1}  for $p\neq 2$. We will do it for $p=2$ in the following.
\end{remark}
The specificity of Hilbert space technique in $F_{\alpha}^{2}$ and an uniform expansion for the functions $f\in F_{\alpha}^{2}$ about the point $z=0$ gives our second main result.

\begin{theorem}
\label{hilbert11}
Let $f\in F_{\alpha}^{2}.$  Then we have the following sharp inequality
$$|f^{(n)}(z)|\leq \sqrt{\alpha^n\Gamma(1 + n) {}_1F_{1}(1 + n; 1; \alpha|z|^2)}\|f-T_{n}(f,0)\|_{2,\alpha},$$
where $n\in \mathbf{N}$ and $T_{n}(f,0)=\sum_{k=0}^{n-1}\frac{f^{k}(0)}{k!}w^{k}.$

The extremal functions are of the form
$$f(w)=\alpha^{n}e^{\alpha \bar{z} w}w^{n}.$$
Therefore
$$|f^{n}(z)|\leq \sqrt{\alpha^n\Gamma(1 + n) {}_1F_{1}(1 + n; 1; \alpha|z|^2)}\inf_{P\in {\mathcal P}_{n}}\|f-P\|_{2,\alpha},$$
where ${\mathcal P}_{n}$ is the set of all (analytic) polynomials of degree at most $n-1.$

In particular for $n=1$ we get the following best estimates
\begin{equation}\label{est2}
\begin{split}|f'(z)|&\le \sqrt{\alpha \left(1+\alpha |z|^2\right)}e^{\alpha  |z|^2/2} \|f(\cdot)-f(0)\|_{2,\alpha}\\&\le \sqrt{\alpha  \left(1+\alpha |z|^2\right)}e^{\alpha  |z|^2/2} \|f\|_{2,\alpha}.\end{split}
\end{equation}

\end{theorem}
One should notice that in  Theorem \ref{main} and  Theorem \ref{hilbert11} the obtained sharp estimates do not coincide (when $p=2$). This at first sight  paradoxical situation can be explained by the fact that the Taylor expansions $T_{n}(f,z)$ and $T_{n}(f,0)$ that appear in Theorem \ref{main} and Theorem \ref{hilbert11} (respectively) differ in general for $z\neq 0.$ However, in case when $z=0$ the obtained estimates match each other.

The following example shows that Theorem~\ref{main} is not true if instead of $\|f-T_n(f,z)\|_{p,\alpha}$ we put $\|f\|_{p,\alpha}$.
\begin{example}
In the following calculations we consider the parameter $s$ to be real and $z\in \mathbb{C},$ $\Re{(z)}\neq 0.$
Consider
\begin{equation*}
\begin{split}
I(s) &=\int_{\mathbb{C}}|(w-z)^{n} e^{\alpha (w-z)\overline{z}}+s|^p e^{-|w|^2 \alpha p/2}dA(w), \mathrm{ i.e. }\\
I(s)&=\int_{\mathbb{C}}|w^{n} e^{\alpha w\overline{z}}+s|^p e^{-|z+w|^2 \alpha p/2}dA(w).
\end{split}
\end{equation*}
Differentiating the last equation with respect to $s$ we get
$$I'(s) = p \int_{\mathbb{C}}(\Re(w^{n} e^{\alpha w\overline{z}})+s)  |w^{n} e^{\alpha w\overline{z}}+s|^{p-2} e^{-|z+w|^2 \alpha p/2}dA(w),$$
and for $s=0,$

$$I'(0) = p \int_{\mathbb{C}}\Re(w^{n} e^{\alpha w\overline{z}})  |w^{n} e^{\alpha w\overline{z}}|^{p-2} e^{-|z+w|^2 \alpha p/2}dA(w).$$
On the other hand, taking into account that $z=|z|e^{i\varphi}$ and $w=re^{it}$ we have
\begin{equation}
\begin{split}
I'(0)&= p \int_{\mathbb{C}} |w|^{n(p-2)}\Re\left(w^{n} \frac{e^{\alpha w\overline{z}}}{|e^{\alpha w\bar{z}}|^{2}}\right)  e^{-(|z|^2+|w|^2) \alpha p/2}dA(w)\\
 &=p \int_{\mathbb{C}} |w|^{n(p-2)}\Re\left(w^{n} e^{-\alpha z\overline{w}}\right) e^{-(|z|^2+|w|^2) \alpha p/2}dA(w)\\
&=pe^{-\frac{p\alpha|z|^{2}}{2}} \int_{0}^{\infty}e^{-\frac{p\alpha r^{2}}{2}}r^{n(p-2)+1}\Re\left(\int_{0}^{2\pi}r^{n}e^{int} e^{-\alpha r|z|e^{i(\varphi-t)}}dt\right)dr\\
&=pe^{-\frac{p\alpha|z|^{2}}{2}} \int_{0}^{\infty}e^{-\frac{p\alpha r^{2}}{2}}r^{n(p-1)+1}\Re\left(\frac{e^{in\varphi}}{i}\int_{|\xi|=1}\xi^{n-1} e^{-\alpha r|z|\xi^{-1}}d\xi\right) dr\\
&=\frac{(-1)^{n}2p\pi\alpha^{n}e^{-\frac{p\alpha|z|^{2}}{2}}\Re{z^{n}}}{n!}  \int_{0}^{\infty}e^{-\frac{p\alpha r^{2}}{2}}r^{np+1} dr\\
&=\frac{(-1)^{n}2\pi\alpha^{n-1} e^{-\frac{p\alpha|z|^{2}}{2}} (\alpha p/2)^{-np/2}}{n!} \Gamma\left(1+\frac{np}{2}\right)\Re{z^{n}}\neq 0.
\end{split}
\end{equation}
Assume that $n=1$.
From this we conclude that $I(0)$ is not minimum of $I(s)$, and therefore for $f(w)=(w-z)^{n} e^{\alpha (w-z)\overline{z}}+s$,  there is a small  enough $s$ so that $\|f\|_{p,\alpha}<\|f(\cdot)-f(z)\|_{p,\alpha}$, but we know that $$|f'(z)|=\frac{ (\alpha p/2)^{1/2} }{ \Gamma^{1/p}\left(1+\frac{ p}{2}\right)}e^{\frac{\alpha |z|^2}{2}}\|f(\cdot)-f(z)\|_{p,\alpha},$$ and so $$|f'(z)|>\frac{ (\alpha p/2)^{1/2} }{ \Gamma^{1/p}\left(1+\frac{ p}{2}\right)}e^{\frac{\alpha |z|^2}{2}}\|f(\cdot)\|_{p,\alpha}.$$ A similar analysis works for general $n$.
\end{example}
At the end of this section we present another type of pointwise estimate which differs quantitative from the previously stated results.  More concretely, in Theorem \ref{nulla} we establish the sharp pointwise estimate for the function $f\in F_{\alpha}^{p}$ in the point $z=0$ in terms of $\|f\|_{p,\alpha},$ while in   Corollary \ref{nulla1} we generalize this estimation for arbitrary $z\in\mathbb{C}.$
\begin{theorem}\label{nulla} Let $f\in F_{\alpha}^{p},$ where $1\leq p<\infty.$ Then
$$|f^{(n)}(0)|\leq \frac{(\alpha p)^{\frac{n}{2}}n!}{2^{\frac{n}{2}}\Gamma^{1/p}\left(1+\frac{np}{2}\right)}\|f\|_{p,\alpha}.$$ The extremal functions are of the form $$f(w)=Aw^{n},$$ where $A$ is arbitrary  complex constant.
\end{theorem}
\begin{remark}
It is easy to conclude that \begin{equation}\label{pp}\|f-f(0)\|_{p,\alpha}\le \|f\|_{p,\alpha},\end{equation} for $p=2$, however the inequality \eqref{pp} does not hold in general if $p\neq 2$.
\end{remark}
\begin{remark}

In \cite{zhu} (exercise 18, pp.90)  the $n$th Taylor coefficient of the function $f\in F_{\alpha}^{p},$ $1\leq p<\infty$  was estimated as follows
  $$|f^{(n)}(0)|\leq C(\alpha,n)\|f\|_{p,\alpha},$$ where $C(\alpha,n)=\left(\frac{\alpha e}{n}\right)^{\frac{n}{2}}n!.$

Let us denote the obtained estimate-constant from Theorem \ref{nulla} by $$C_{p}(\alpha,n)=\frac{(\alpha p)^{\frac{n}{2}}n!}{2^{\frac{n}{2}}\Gamma^{1/p}\left(1+\frac{np}{2}\right)}.$$ Then
$$\frac{C(\alpha,n)}{C_{p}(\alpha,n)}=\frac{(2e)^{n/2}\Gamma^{1/p}\left(1+\frac{np}{2}\right)}{(np)^{n/2}}.$$

 Moreover,  using the asymptotic behaviour of the  Gamma function $\Gamma(1+x)\sim \sqrt{2\pi x}\left(\frac{x}{e}\right)^{x},$ $x\rightarrow+\infty$ we get
  $$\frac{C(\alpha,n)}{C_{p}(\alpha,n)}\sim 1,\enspace \mbox{as}\enspace p\rightarrow +\infty$$
  or $$\frac{C(\alpha,n)}{C_{p}(\alpha,n)}\sim (\pi n p)^{\frac{1}{2p}},\enspace\mbox{as}\enspace n\rightarrow+\infty.$$
  Here, the relation $a_{n}\sim b_{n}, (a_{n},b_{n}>0)$ $n\rightarrow+\infty$ means $\lim_{n\rightarrow+\infty}\frac{a_{n}}{b_{n}}=1.$
\end{remark}
\begin{corollary}
\label{nulla1}
  Let $f\in F_{\alpha}^{p},$ where $1\leq p<\infty$ and $z\in\mathbb{C}.$ Then
$$|f^{(n)}(z)|\leq \frac{(a p)^{\frac{n}{2}}n!}{2^{\frac{n}{2}}\Gamma^{1/p}\left(1+\frac{np}{2}\right)}\|f_{z}\|_{p,\alpha},$$
where $f_{z}(w)=f(z+w).$ The extremal functions are of the form $$f(w)=A(w-z)^{n},$$ where $A$ is some complex constant.
\end{corollary}
\begin{proof}
The proof follows directly by applying  Theorem \ref{nulla} on the function $F(w)=f(z+w).$
\end{proof}

Proofs of the results are given in  sections~3,~4 and ~5. In the section~6 we consider the differential operator and its norm between two Fock spaces with different weights. We also consider the shift operator.

\section{Proof of  Theorem \ref{main}}
\begin{proof}
At the beginning we will examine the case when $n=1.$

Let $$\varphi(t)= \frac{f(t+z)-f(z)-f'(z) t e^{\alpha t \bar z}}{t^2}.$$ Then $\varphi$ is an entire analytic function and we have
\begin{equation}\label{prez}f(w) = f(z)+ f'(z)(w-z) e^{\alpha(w\bar z-|z|^2)}+(w-z)^2 \varphi(w-z).\end{equation}

Then,
\begin{equation*}
\begin{split}
&\|f-f(z)\|_{p,\alpha}^{p}\\
&=\frac{p\alpha}{2\pi} \int_{C} |f'(z)(w-z) e^{\alpha(w-z)\bar z}+(w-z)^2 \varphi(w-z)|^p e^{-\alpha p |w|^2/2} dA(w)\\
&=\frac{p\alpha}{2\pi} \int_{C} |f'(z)w e^{\alpha w\bar z}+w^2 \varphi(w)|^p e^{-\frac{\alpha p |w+z|^2}{2}} dA(w)
\end{split}
\end{equation*}
So for $w=re^{it}$, we get

 \[\begin{split}\label{nizovi}&\|f-f(z)\|_{p,\alpha}^p\\
&=\frac{p\alpha}{2\pi} \int_0^\infty \int_0^{2\pi} r|re^{it}f'(z)e^{\alpha re^{it} \bar z}+r^2e^{2it}\varphi(re^{it})|^p e^{-\frac{\alpha p|re^{it}+z|^2}{2}} dt dr\\&= \frac{p\alpha}{2\pi} \int_0^\infty e^{-\frac{p\alpha r^{2}}{2}}r^{p+1}\int_{\mathbf{T}}  \left|f'(z)+e^{-\alpha (r\zeta)
\bar z}r\zeta\varphi(r\zeta )\right|^p|d\zeta| e^{-\frac{p\alpha|z|^2}{2}} dr\\&\ge p\alpha e^{-\frac{p\alpha|z|^{2}}{2}}\int_0^\infty  e^{-\frac{p\alpha r^{2}}{2}}r^{p+1} |f'(z)|^p  dr \\&=2^{p/2} (\alpha p)^{-p/2} \Gamma\left(1+\frac{p}{2}\right)e^{-\frac{p\alpha|z|^2}{2}}|f'(z)|^{p}\\&=|f'(z)|^{p}\|(w-z) e^{\alpha (w-z)\overline{z}}\|^p_{p,\alpha}.\end{split}\]
The next-to-the last inequality was obtained  by using the subharmonicity of the function $$F(\zeta)=\left|f'(z)+e^{-\alpha (r\zeta)
\bar z}r\zeta\varphi(r\zeta )\right|^p$$ in $|\zeta|\leq 1.$

From the last sequence of inequalities we conclude that the equality is going to be attained for the functions $f\in F_{\alpha}^{p}$ for which $\varphi=0$ in the representation \eqref{prez} and that the extremal functions are given by $$f(w)=A(w-z)e^{\alpha(w-z)\bar{z}},$$ where $A$ is some constant.

The proof for $n>1$  is similar to the case $n=1$. We give its brief outline.
For $f\in F_{\alpha}^{p}$ we have the following representation
$$f(w)=\sum_{k=0}^{n-1}\frac{f^{(k)}(z)}{k!}(w-z)^{k}+\frac{f^{(n)}(z)}{n!}(w-z)^{n}e^{\alpha(w-z)\bar z}+(w-z)^{n+1}\varphi(w-z),$$
where $\varphi$ is a certain entire function depending from $f.$

In a similar manner as it was done before, we have the inequality
\begin{equation}
\begin{split}
\|f-T_{n}(f,z)\|_{p,\alpha}^{p}&\geq \left|\frac{f^{n}(z)}{n!}\right|^{p}\|(w-z)^{n}e^{\alpha(w-z)\bar z}\|_{p,\alpha}^{p}\\
&=\frac{p\alpha}{2\pi} \left|\frac{f^{n}(z)}{n!}\right|^{p}\int_{\mathbb{C}}|w|^{pn}e^{\alpha p \Re{(w\bar{z})}}e^{-\frac{p\alpha|z+w|^{2}}{2}}dA(w)\\
&=\left(p\alpha\int_{0}^{\infty}r^{pn+1}e^{-\frac{p\alpha r^{2}}{2}}dr\right)e^{-\frac{p\alpha|z|^{2}}{2}}\left|\frac{f^{n}(z)}{n!}\right|^{p}\\
&=2^{\frac{np}{2}}(\alpha p)^{-\frac{np}{2}}\Gamma\left(1+\frac{np}{2}\right)e^{-\frac{p\alpha|z|^{2}}{2}}\left|\frac{f^{n}(z)}{n!}\right|^{p},
\end{split}
\end{equation} where the equality is attained again if $\varphi=0.$

\end{proof}

\section{Proof of Theorem \ref{hilbert11}}

\begin{proof}
Firstly  we will present the proof in case when $n=1.$

Using the  uniform expansion of the function $f\in F_{\alpha}^{2},$
$$f(w)=\sum_{k=0}^{\infty}a_{k}w^{k},$$ and by using the polar coordinates we easily get
\begin{equation}
\begin{split}
\|f\|_{2,\alpha}^{2}&=\frac{\alpha}{\pi}\int_{\mathbb{C}}\left|\sum_{k=0}^{\infty}a_{k}w^{k}\right|^{2}e^{-\alpha|w|^{2}}dA(w)\\
&=\frac{\alpha}{\pi}\sum_{k=0}^{\infty}|a_{k}|^{2}\int_{0}^{2\pi}dt\int_{0}^{\infty}e^{-\alpha r^{2}}r^{2k+1}dr\\
&=\sum_{k=0}^{\infty}|a_{k}|^{2}\frac{ k!}{\alpha^{k}}.
\end{split}
\end{equation}

Since $f'(w) = \sum_{n=1}^\infty n a_n w^{n-1},$ the Cauchy-Schwarz inequality implies

$$\left| \sum_{n=0}^\infty (n+1) a_{n+1} w^{n}\right|^2\leq\left( \sum_{n=0}^\infty \frac{(n+1)! |a_{n+1}|^2 }{\alpha^{n+1}}\right)\left(\sum_{n=0}^\infty \frac{(n+1)\alpha^{n+1} |w|^{2n}}{n!}\right),$$

with an equality for certain $w=z$ if and only if $$a_{n+1}=\frac{\alpha^{n+1} \bar{z}^n}{n!}.$$

 Then $$f(w) = \sum_{n=0}^\infty a_n w^n=a_0 +\alpha e^{\alpha w \bar{z}} w.$$

In case we have an optimal condition $f(0)=0,$ then $a_0 = -\alpha ze^{\alpha |z|^2}. $

Thus $$f'(z)=\alpha e^{\alpha  |z|^{2}} (1 + \alpha  |z|^{2}).$$

Further

$$\sum_{n=0}^\infty \frac{(n+1)\alpha^{n+1} |z|^{2n}}{n!}=\alpha e^{\alpha |z|^2} \left(1+\alpha |z|^2\right)$$

and

$$ \sum_{n=0}^\infty \frac{(n+1)! |a_{n+1}|^2 }{\alpha^{n+1}}=\alpha e^{\alpha |z|^2} \left(1+\alpha |z|^2\right).$$
If  we consider the case when $n>1,$ the $n$th derivative is then given by the following formula
$$f^{(n)}(w) = \sum_{k=n}^\infty \frac{\Gamma(k+1)}{\Gamma(k-n+1)} a_k w^{k-n}.$$
Repeating the previous approach conducted for the case $n=1,$ we may conclude that the coefficients of the extremal function have the following form $$a_{k+n}=\frac{\alpha^{k+n} \bar{z}^k}{\Gamma(k+1)}.$$
Thus
\begin{equation}
\begin{split}
f(w) &= \sum_{k=0}^\infty a_k w^k=\sum_{k=0}^{n-1}a_{k}w^{k} + \sum_{k=0}^{\infty}\frac{\alpha^{n+k}\bar{z}^{k}}{\Gamma(k+1)}w^{k+n}\\
&=\sum_{k=0}^{n-1}a_{k}w^{k}+\alpha^{n}w^{n}e^{\alpha w\bar{z}}.
\end{split}
\end{equation}

The $n$th derivative is then given by
 $$f^{(n)}(w)= \sum_{k=0}^{\infty}\frac{\alpha^{n+k}\bar{z}^{k}\Gamma(k+n+1)}{\Gamma^{2}(k+1)}w^{k},$$
and
$$f^{(n)}(z)=\sum_{k=0}^{\infty}\frac{\alpha^{n+k}\Gamma(k+n+1)}{\Gamma^{2}(k+1)}|z|^{2k}.$$

Therefore,
\begin{equation}
\begin{split}
\sum_{n=0}^\infty \frac{(n+k)! |a_{n+k}|^2}{\alpha^{k+n}}&=\sum_{k=0}^{\infty}\frac{\alpha^{n+k}\Gamma(k+n+1)}{\Gamma^{2}(k+1)}|z|^{2k}\\
&=\alpha^n\Gamma(1 + n) {}_1F_{1}(1 + n; 1; \alpha|z|^2).\end{split}\end{equation}

Let us summarize the obtained results in the next inequality

$$|f^{n}(z)|\leq \sqrt{\alpha^n\Gamma(1 + n) {}_1F_{1}(1 + n; 1; \alpha|z|^2)}\inf_{{\mathcal P}_{n}}\|f-p\|_{2,\alpha},$$
where ${\mathcal P}_{n}$ is the set of all (analytic) polynomials of degree at most $n-1.$

On the other hand, $$\inf_{{\mathcal P}_{n}}\|f-p\|_{2,\alpha}=\|f-T_{n}(f,\cdot)\|_{2,\alpha},$$ where
$T_{n}(f,w)=\sum_{k=0}^{n-1}\frac{f^{k}(0)}{k!}w^{k}$ is a Taylor expansion of the function $f$ about point $w=0$ up to  order $n-1.$
\end{proof}

\section{Proof of Theorem \ref{nulla}}

\begin{proof}

For $f\in F_{\alpha}^{p}$ by using the Taylor expansion of the function $f(z)$ in $z=0,$ it is clear that the function $f$ can be presented as follows
\begin{equation}\label{prezen}f(z)=\sum_{k=0}^{n-1}\frac{f^{(k)}(0)}{k!}z^{k}+\frac{f^{(n)}(0)}{n!}z^{n}+z^{n+1}\varphi(z),\end{equation}
 where $\varphi$ is certain entire function. Let us first suppose that $\frac{f^{(n)}(0)}{n!}=1.$

 Keeping in mind the identity \eqref{prezen} we have
\begin{equation}
\begin{split}
\label{niz}
&\|f\|^p_{p,\alpha}\\
& =\frac{p\alpha}{2\pi} \int_{\mathbf{C}}|f(z)|^p e^{-\alpha p |z|^2/2}dA(z)\\
&= \frac{p\alpha}{2\pi}\int_{\mathbb{C}}  \left|\sum_{k=0}^{n-1}\frac{f^{(k)}(0)}{k!}z^{k}+\frac{f^{(n)}(0)}{n!}z^{n}+z^{n+1}\varphi(z)\right|^{p} e^{-\alpha p |z|^2/2}dA(z)\\
&= \frac{p\alpha}{2\pi}\int_{0}^{\infty}e^{-\frac{\alpha r^{2}}{2}}r^{np+1}\int_{0}^{2\pi}\left|1+\sum_{k=0}^{n-1}\frac{f^{(k)}(0)}{r^{n-k}k!}e^{i(k-n)t}+re^{it}\varphi(re^{it})\right|^{p}dt dr\\
&=\frac{p\alpha}{2\pi}\int_{0}^{\infty}e^{-\frac{\alpha r^{2}}{2}}r^{np+1}\int_{|\xi|=1}\left|1+\sum_{k=0}^{n-1}\frac{f^{(k)}(0)}{r^{n-k}k!}\bar{\xi}^{n-k}+r\xi\varphi(r\xi)\right|^{p}|d\xi|dr\\
&\ge p\alpha \int_{0}^{\infty}e^{-\frac{\alpha r^{2}}{2}}r^{np+1}|1|^{p}dr dt=\|z^{n}\|^p_{p,\alpha}.
\end{split}
\end{equation}

Clearly, in the last inequality we used the subharmonicity of the function $$F(\xi)=\left|1+\sum_{k=0}^{n-1}\frac{f^{(k)}(0)}{r^{n-k}k!}\bar{\xi}^{n-k}+r\xi\varphi(r\xi)\right|^{p}$$ in the unit disc $|\xi|\leq 1,$ and the inequality  $2\pi F(0)\leq\int_{|\xi|=1} F(\xi)|d\xi|.$

Let $$C_p=\|z^{n}\|_{p,\alpha}=\left(\frac{p\alpha}{2\pi}\int_{\mathbf{C}}|z|^p e^{-p\alpha |z|^2/2}dA(z)\right)^{1/p}.$$ Then $$C_p=\left(2^{\frac{np}{2}}  (\alpha p)^{-\frac{np}{2}} \Gamma\left(\frac{np}{2}+1\right)\right)^{\frac{1}{p}}.$$

Further, if $\left|\frac{f^{n}(0)}{n!}\right|=R$, then $\left|\frac{g^{(n)}(0)}{n!}\right|=1$, where $g(z) = \frac{f(z)}{R}$. Further, $\|g\|_{p,\alpha}\ge \|z^{n}\|_{p,\alpha}$, and therefore $$ \|f\|_{p,\alpha}\ge R \|z\|_{p,\alpha}=\left|\frac{f^{(n)}(0)}{n!}\right|\|z^{n}\|_{p,\alpha}.$$ The result follows.

\end{proof}

\section{Some estimates for the differential operator and for the shift operator}

\subsection{The differential operator}
The following example
$$f(z) = \sum_{k=2}^\infty \frac{1}{\sqrt{k(k-1)}} \, \frac{z^k}{\sqrt{k!}}, $$
shows that there are examples of  $F^2_1,$ so that $f'\notin F^2_1.$ However, if $\beta>\alpha$, then $f\in F^2_\alpha$ implies $f'\in F^2_\beta$. Namely
if $f(z)=\sum_{k=0}^\infty a_k z^k$, then $$\|f\|^2_{2,\alpha}=\sum_{k=0}^\infty \frac{|a_k|^2 k!}{\alpha^k}\text{     and     }
\|f'\|^2_{2,\beta}=\sum_{k=1}^\infty \frac{|k a_k|^2 (k-1)!}{\beta^k}.$$
Now the inequality $$ \frac{|k a_k|^2 (k-1)!}{\beta^k}\le C^2 \frac{|a_k|^2 k!}{\alpha^k}$$ is equivalent with $ k\le C^2 (\beta/\alpha)^{k}$.

For $h(x) = \gamma^{-x} x$, we have $$C^2=\max\{h(n): n\in\mathbf{N}\}=\max\{h(m), h(1) \}$$ where $\gamma=\frac{\beta}{\alpha}$ and
$m=\left[\frac{1}{e\log \gamma}\right]$, where $\left[\cdot\right]$ stands for the integer part.

Now we have the following inequality $$ k\le C^2(\beta/\alpha)^{k}$$ for every $k$ and therefore \begin{equation}\label{seek}\|f'\|_{2,\beta}\le C  \|f\|_{2,\alpha}.\end{equation}

The extremal function is $$f(z) = \left\{
                                    \begin{array}{ll}
                                      z,  \hbox{if $h(m)<h(1)$;} \\
                                      z^m, & \hbox{otherwise.}
                                    \end{array}
                                  \right.$$

More general for every positive number $(p\ge 1)$  if $\beta>\alpha$, then $f\in F^p_\alpha$ implies $f'\in F^p_\beta$.  By using the Cauchy formula $$f'(z) = \frac{1}{2\pi i}\int_{|w|=1}\frac{f(w+z)}{w^2} dw$$ and therefore by using the Fubini's theorem and Jensen's inequality we get

\[\begin{split}\int_{\mathbf{C}} |f'(z)|^p e^{-\beta p|z|^2/2} dxdy &=\int_{\mathbf{C}} \left|\frac{1}{2\pi i}\int_{|w|=1}\frac{f(w+z)}{w^2} dw\right|^p e^{-\beta p|z|^2/2} dxdy\\ &\le \frac{1}{2\pi }\int_{\mathbf{C}} \int_{|w|=1}|f(w+z)|^p |dw| e^{-\beta p|z|^2/2} dxdy
\\ &= \frac{1}{2\pi }\int_{|w|=1}\left(\int_{\mathbf{C}} |f(w+z)|^p  e^{-\beta p|z|^2/2} dxdy\right) |dw|
\end{split}
\]
It remains to check the following strighforward inequality $$ e^{-\beta p|z|^2/2}\le c^p e^{-\alpha p|z+w|^2/2},$$ with $$c=e^{\frac{\alpha \beta }{2 \beta-2 \alpha}}$$ for $z\in\mathbf{C}$ and $|w|=1$. Moreover $$\int_{\mathbf{C}} |f(w+z)|^p  e^{-\alpha p|z+w|^2/2} dxdy=\int_{\mathbf{C}} |f(z)|^p  e^{-\alpha p|z|^2/2} dxdy.$$
Therefore $$\|f'\|_{p,\beta}\le e^{\frac{\alpha \beta }{2 \beta-2 \alpha}}\|f\|_{p,\alpha}.$$
Thus,  if $\beta>\alpha$, then $f\in F^p_\alpha$ implies $f'\in F^p_\beta$.

In a similar way we can prove that,  if $\beta>\alpha$, then $f\in F^p_\alpha$ implies $f^{(n)}\in F^p_\beta$ for every integer $n$.
\begin{remark}
For $p\ge 1$, we proved that the differential operator $$D=D_{\alpha,\beta}: F_\alpha^p\to F_\beta^p,$$ defined by $D[f]=f'$ is bounded, provided that $\beta>\alpha$. For $p=2$ we proved that $$\|D\|_2^2 = \max\left\{\gamma^{-1}, \gamma^{-\left[\frac{1}{e\log \gamma}\right]} \left[\frac{1}{e\log \gamma }\right]\right\}, \ \ \ \gamma=\frac{\beta}{\alpha}.$$ It would be of interest to find its norm for general $p\neq 2$.
\end{remark}
In the following we determine the adjoint operator
$$D^*:F^2_\beta \longrightarrow F^2_\alpha,$$
where $\beta >\alpha >0.$

 For this purpose we denote by $(.,.)_\alpha $ the inner product in $F^2_\alpha $ and by $(.,.)_\beta $ the inner product in $F^2_\beta. $
Let $(e_k)_{k \ge 0}$ denote the orthonormal basis $(e_k = \frac{z^k}{c_k})_{k \ge 0} $ of $F^2_\alpha,$ where
$$c_k^2 = \| z^k \|_{2,\alpha} ^2 = \frac{k!}{\alpha^k},$$
and let $(E_k)_{k \ge 0}$ denote the orthonormal basis $(E_k = \frac{z^k}{d_k})_{k \ge 0}  $ of $F^2_\beta,$ where
$$d_k^2 = \| z^k \|_{2,\beta} ^2 = \frac{k!}{\beta^k}.$$
We write $f\in F^2_\alpha$ in the form $f= \sum_{k=0}^\infty f_k e_k,$ where $(f_k)_k \in l^2,$ and
$g\in F^2_\beta$ in the form $g= \sum_{k=0}^\infty g_k E_k,$ where $(g_k)_k \in l^2.$ Then we have
\begin{eqnarray*}
(Df,g)_\beta &=& (\sum_{k\ge 1} kf_k \frac{z^{k-1}}{c_k}, g)_\beta \\
&=& (\sum_{k\ge 0} (k+1)f_{k+1} \frac{d_k}{c_{k+1}} E_k, \sum_{k\ge 0} g_k E_k)_\beta \\
&=& \sum_{k\ge 0} (k+1)f_{k+1} \frac{d_k}{c_{k+1}} \overline g_k\\
&=& \sum_{k\ge 1} f_{k} k\frac{d_{k-1}}{c_{k}} \overline g_{k-1}\\
&=& (f,D^*g)_\alpha,
\end{eqnarray*}
which implies that
$$D^*g = \sum_{k\ge 1}  k\frac{d_{k-1}}{c_{k}}  g_{k-1} e_k.$$
Notice that
$$zg(z)= \sum_{k\ge 1} \frac{c_k}{d_{k-1}} g_{k-1}e_k.$$
\begin{remark} The adjoint operator $D^*$ is not of the form $zg(z),$ as it is for the unbounded densely defined operator $D : F^2_1 \longrightarrow F^2_1,$ see \cite{foland}, \cite{has}.
In this caes the operator $D$ appears as the annihilation operator and its adjoint as the creation operator in quantum mechanics.
\end{remark}
\vskip 0.5 cm

\subsection{The shift operator}

The shift (or multiplication) operator $M,$ defined by $Mf(z)=zf(z),$ is bounded as an operator from $F^2_\alpha$ to $,F^2_\beta$ for $\beta>\alpha.$ This follows from a similar reasoning as for the operator $D.$

Let $f\in F^2_\alpha.$ If $f(z)=\sum_{k=0}^\infty a_k z^k$, then $$\|f\|^2_{2,\alpha}=\sum_{k=0}^\infty \frac{|a_k|^2 k!}{\alpha^k}\text{     and     }
\|zf\|^2_{2,\beta}=\sum_{k=0}^\infty \frac{| a_k|^2 (k+1)!}{\beta^{k+1}}.$$
Now the inequality $$ \frac{| a_k|^2 (k+1)!}{\beta^{k+1}}\le C^2 \frac{|a_k|^2 k!}{\alpha^k}$$ is equivalent with $ k\le C^2\alpha (\beta/\alpha)^{k}$. Hence we get
$$\|M\|_2^2 = \alpha \|D\|_2^2.$$

For $f\in F^2_\alpha$ and $g\in F^2_\beta$ we get
\begin{eqnarray*}
(Mf,g)_\beta &=& (\sum_{k\ge 0} f_k \frac{z^{k+1}}{c_k} , g)_\beta\\
&=& (\sum_{k\ge 0} f_k \frac{d_{k+1}}{c_k}E_{k+1} , g)_\beta \\
&=& \sum_{k\ge 0} f_k \frac{d_{k+1}}{c_k} \overline g_{k+1}\\
&=& (f, M^*g)_\alpha,
\end{eqnarray*}
which implies that
$$M^*g = \sum_{k\ge 1}  \frac{d_{k}}{c_{k-1}}  g_{k} e_k.$$

\end{document}